\documentclass[12pt,article]{amsart}
\usepackage[top=1.15in, bottom=1.15in, left=1.17in, right=1.17in]{geometry}
\usepackage{amsthm, amsmath, amssymb, amscd, latexsym, multicol, verbatim, enumerate, graphicx,xy, color, floatrow, caption, subcaption}
\usepackage{tikz}
\usepackage{float}
\usepackage[colorlinks=true, pdfstartview=FitV, linkcolor=blue, citecolor=blue, urlcolor=blue, breaklinks=true]{hyperref}
\usepackage[english]{babel}
\usepackage[T1]{fontenc}
\usepackage[utf8]{inputenc}
\usepackage{mathrsfs}
\usepackage{ytableau}
\usepackage{cleveref}
\usepackage{tikz-cd}
\usepackage{filecontents}

\usepackage[backend=bibtex,style=alphabetic,doi=false,isbn=false,url=false,giveninits=true,maxbibnames=50]{biblatex}
\usepackage[colorinlistoftodos]{todonotes}

\newcounter{cnt}
 \makeatletter
\def\mydggeometry{\makeatletter\dg@YGRID=1\dg@XGRID=20\unitlength=0.003pt\makeatother}
\makeatother

\newtheorem{theorem}{Theorem}[section]
\newtheorem{lemma}[theorem]{Lemma}
\newtheorem{corollary}[theorem]{Corollary}
\newtheorem{proposition}[theorem]{Proposition}

\theoremstyle{definition}
\newtheorem{definition}[theorem]{Definition}
\newtheorem{example}[theorem]{Example}
\newtheorem{remark}[theorem]{Remark}

\newcommand{\g}{\mathfrak{g}}

\newcommand{\Z}{\mathbb{Z}}
\newcommand{\C}{\mathbb{C}}
\newcommand{\R}{\mathbb{R}}
\newcommand{\ttt}{\textsf{t}}

\newcommand{\sss}{\textsf{s}}

\newcommand{\G}{{\operatorname*{G}}}

\newcommand{\lie}{\mathfrak}

\newcommand{\U}{\mathbb{U}}
\newcommand{\supp}{\text{supp }}

\begin{document}
\author[G. Balla, G. Fourier, and K. Kambaso]{George Balla, Ghislain Fourier, and Kunda Kambaso}
\title[Demazure modules in types A and C]{PBW filtration and monomial bases for Demazure modules in types A and C}

\address{Chair of Algebra and Representation Theory, RWTH Aachen University, Pontdriesch 10-16, 52062 Aachen, Germany}
\email{balla@art.rwth-aachen.de}
\email{fourier@art.rwth-aachen.de}
\email{kambaso@art.rwth-aachen.de}

\maketitle
\begin{abstract}
    We characterise the symplectic Weyl group elements such that the FFLV basis is compatible with the PBW filtration on symplectic Demazure modules, extending type {\tt A} results by the second author. Surprisingly, the number of such elements depends not on the type {\tt A} or {\tt C} of the Lie algebra but on the rank only.\\
    \textbf{Keywords}: PBW filtration; FFLV basis; Demazure modules
\end{abstract}

\section{Introduction} 
Let $\lie g$ be a complex simple, finite-dimensional Lie algebra. For a dominant integral weight $\lambda$, let $V(\lambda)$ denote the simple, finite-dimensional $\lie g$-module of highest weight $\lambda$ and let $v_\lambda\in V(\lambda)$ be a fixed highest weight vector. Let $\text{gr } V(\lambda)$ denote the associated graded module with respect to the PBW filtration on $V(\lambda)$. These degenerate modules have been studied extensively in the last decade, mainly also due to the induced degenerate flag variety, defined by Feigin in \cite{Fei12} as a highest weight orbit of an action of a degenerate group on $\text{gr } V(\lambda)$. One of the achievements of this framework so far is a ``new'' monomial basis of $\text{gr } V(\lambda)$ and hence of $V(\lambda)$ for $\lie{sl}_n, \lie{sp}_{2n}$ \cite{FFL11a, FFL11b}. These monomials are naturally assigned with lattice points in a normal polytope $P(\lambda)\subset \mathbb{R}_{\geq 0}^{N}$, $N$ the number of positive roots, known as the FFLV polytope. 

We seek to understand if this basis is compatible with Demazure modules. The second author has provided in \cite{Fou16} the compatibility for Demazure modules in type {\tt A} corresponding to a class of Weyl group elements that avoid the patterns $4231$ and $2413$. These are called triangular Weyl group elements. Formulated differently, a particular face of the FFLV polytope indexes a basis of the graded Demazure modules for triangular Weyl group elements. The PBW-degenerated Schubert varieties in that context were also studied. The main goal of the current paper is to extend these results to a more general setting that also includes Demazure modules and Schubert varieties in type {\tt C}.\\

Let $R^-$ denote the set of negative roots of ${\lie g}$. We call a subset $A \subset R^-$, \emph{FFLV-admissible} if it is closed under taking sums inside $R^-$ and stable for a \textit{join} operation (see Definition \ref{def:admissible}). This join operation allows for a straightening law on monomials. For example, the sets corresponding to the triangular subsets studied in \cite{Fou16} are admissible but admissibility is not restricted to triangular subsets, as for example, in type ${\tt A}$ every single element set is admissible. We consider a submodule of $V(\lambda)$ corresponding to an FFLV-admissible set $A$, which we denote by $V_A(\lambda)$, generated through $v_{\lambda}$ by the subalgebra of $\lie g$ defined by $A$. It turns out that the induced PBW filtration is compatible with the one on $V(\lambda)$ (Proposition \ref{pro:pbwsubmodules}). Let $\text{gr } V_A(\lambda)$ denote the associated graded submodule.

Our first interest is in constructing a monomial basis for $V_{A}(\lambda)$ and $\text{gr } V_{A}(\lambda)$. For this, we will consider the following faces of the polytope $P(\lambda)$: for $A\subset R^-$, let $P_{A}(\lambda)$ denote the face of $P(\lambda)$ obtained by setting the coordinates $s_{\beta}=0$ for all $\beta \notin A$. Let $S_{A}(\lambda)$ denote the set of lattice points in $P_{A}(\lambda)$. We prove: 

\begin{theorem}\label{thm:intro-main}
Let $\lie g$ be $\lie{sl}_{n+1}(\C)$ or $\lie{sp}_{2n}(\C)$ and let $A\subset R^-$ be FFLV-admissible, then:
\begin{enumerate}
    \item[(a)] for all $\lambda, \mu$, dominant integral weights, $S_{A}(\lambda) + S_{A}(\mu)= S_{A}(\lambda+\mu)$.
    \item[(b)] the set $S_{A}(\lambda)$ parametrizes a monomial basis of the PBW-graded submodule $\text{gr } V_{A}(\lambda)\subset \text{gr } V(\lambda)$ and hence of $V_{A}(\lambda)$. 
\end{enumerate}
\end{theorem}
Our approach in proving part (b) of the above theorem relies on having defining relations for the modules $\text{gr } V(\lambda)$ as constructed in \cite{FFL11a, FFL11b} (see Theorem \ref{thm:main}). 
A direct consequence is the following and explains our interest in considering FFLV-admissible subsets: if $A\subset R^-$ is not FFLV-admissible and $\lambda$ is regular, then $S_{A}(\lambda)$ doesn't parametrize (in a ``natural'' way) a monomial basis of $\text{gr } V_{A}(\lambda)$ (Lemma \ref{lem:notadmissible}). We should note, that a Kogan face construction as in \cite{Kog00, KST12, BF15} is not applicable, as in contrast to that case, the coordinates are assigned positive roots and not simple roots (with respect to a reduced expression of the longest Weyl group element).

We would also like to note that the polytopes constructed in \cite{Kam20} parametrizing monomial bases for submodules studied there, are in general not faces of the FFLV polytope $P(\lambda)$ in types {\tt A} and {\tt C}. On the other hand, monomial bases similar to those discussed here have been constructed in \cite{Mak19} for orthogonal modules in type {\tt B}.  However, our approach is not applicable in that situation, since defining relations for type {\tt B} are still missing.

It turns out that the submodules $V_{A}(\lambda)$ are favourable in the sense of \cite{FFL17b}, so we obtain flat PBW-degenerated symplectic Schubert varieties which degenerate further into toric varieties (Proposition \ref{pro:degenerations}). Moreover, we consequently describe a monomial basis for the homogeneous coordinate rings of the PBW-degenerate Schubert varieties labelled by PBW-semistandard tableaux (Proposition \ref{pro:permissibletableaux}).\\

Our paper is organised as follows: in Section~\ref{sec:preliminaries}, we recall basic notation on Lie algebras while in Section~\ref{sec:admissible},
we introduce FFLV-admissible subsets and show that they are compatible with Dynkin automorphisms. In Section~\ref{sec:triangular}, we describe the correspondence between FFLV-admissible subsets and triangular Weyl group elements and in Section~\ref{sec:application} we describe the monomial basis for our submodules and give the geometric interpretation in terms of PBW-degenerated Schubert varieties.

\subsection*{Acknowledgements}
The first and third authors are supported by the Deutscher Akademischer Austauschdienst (DAAD, German Academic Exchange Service) scholarship program: Research Grants - Doctoral Programs in Germany [program-ID 57440921]. The work of the second author is funded by the Deutsche Forschungsgemeinschaft (DFG, German Research Foundation) through [project-ID 286237555 – TRR 195].

\section{Preliminaries: notation on Lie algebras}\label{sec:preliminaries}
We follow \cite{Hum12}. Let $\lie g$ be the simple, complex Lie algebra of classical type ${\tt A}_n$, ${\tt B}_n$, ${\tt C}_n$ or ${\tt D}_n$.
We fix a triangular decomposition $\lie g = \lie n^+ \oplus \lie h \oplus \lie n^-$. Having fixed this decomposition, we denote by $\{\alpha_1, \ldots, \alpha_n\}$, the set of simple roots, $R^+$ and $R^-$, the sets of positive roots and negative roots respectively.  
\begin{itemize}
    \item For $\lie g = \lie{sl}_{n+1}$ (of type {\tt A}$_n$) and $1\leq i \leq j\leq n$, we abbreviate $\alpha_{i,j} := \alpha_i + \cdots + \alpha_j.$ Then $R_{{\tt A}_n}^+= \{\alpha_{i,j} \mid 1\leq i \leq j\leq n\}$. 
    \item For $\lie g = \lie{so}_{2n+1}$ (of type {\tt B}$_n$), $\alpha_{i,j}$ as above and $\alpha_{i, \overline{j}} := \alpha_ i + \cdots + \alpha_{j-1} + 2 \alpha_{j} + \cdots + 2 \alpha_n$, we have $R_{{\tt B}_n}^+= \{\alpha_{i,j}, \alpha_{k,\overline{l}} \mid 1\leq i \leq j\leq n, \,\, 1\leq k<l\leq n\}$. 
    \item For $\lie g = \lie{sp}_{2n}$ (of type {\tt C}$_n$), $\alpha_{i,j}$ as above and $\alpha_{i, \overline{j}} := \alpha_ i + \cdots + \alpha_{j-1} + 2 \alpha_j + \cdots + 2 \alpha_{n-1} + \alpha_n$ (with $\alpha_{i,n}= \alpha_{i,\overline{n}}$), one has $R_{{\tt C}_n}^+= \{\alpha_{i,j}, \alpha_{i,\overline{j}} \mid 1\leq i \leq j\leq n\}$. 
    \item Finally, for $\lie g = \lie{so}_{2n}$ (of type {\tt D}$_n$), $\alpha_{i,j}$ as before, $\alpha_{i, \overline{j}} := \alpha_ i + \cdots + \alpha_{n-2} + \alpha_{j} + \cdots + \alpha_n,$ the set of positive roots is given by $R_{{\tt D}_n}^+= \{\alpha_{i,j}, \alpha_{k,\overline{l}} \mid 1\leq i \leq j\leq n-1, \,\, 1\leq k<l\leq n\}$.
\end{itemize}
For an element $\beta = \sum a_i \alpha_i \in \sum \mathbb{Z} \alpha_i$ we denote $\supp \beta = \{ i \mid a_i \neq 0\}$ and further say $\alpha$ and $\beta$ are linked, if $\supp \alpha + \beta$ defines a connected subdiagram of the Dynkin diagram.

For $\beta\in R^+$, let $e_{\beta}$ and $f_{\beta}$ denote fixed basis elements in the root spaces of $\lie g$ of weights $\beta$ and $-\beta$ respectively. We will use short notation $e_{i,j}$ resp. $f_{i, \overline{j}}$ for $e_\beta$ and $f_\beta$ where $\beta = \alpha_{i,j}$ or $\beta = \alpha_{i, \overline{j}}$. For $\beta\in R^+$, $h_{\beta}$ will denote the corresponding co-root.\\

Let the set of (dominant) integral weights of $\lie g$ be denoted by $P$ (resp. $P^+$). The fundamental weights are denoted $\{ \omega_1, \ldots, \omega_n\}$, they generate the monoid $P^+$. 
For $\lambda \in P^+$,  we denote by $V(\lambda)$, the simple, finite-dimensional highest weight module of highest weight $\lambda$, with a highest weight vector $v_\lambda$.\\
For Lie algebra $\lie a$, let $U(\lie a)$ denote the universal enveloping algebra. Let $(z_1, \ldots, z_l)$ be an ordered basis of $\lie a$, then the PBW filtration on $U(\lie a)$ is defined as
\[
U(\lie a)_s := \langle z_{i_1} \cdots z_{i_r} \mid r \leq s, z_{i_j} \in \lie a\rangle_{\mathbb{C}}.
\]
The famous PBW theorem states that the associated graded algebra is the symmetric algebra $S(\lie a)$. On a cyclic $U(\lie a)$-module $M$ with generator $m \in M$, one has the induced filtration $M_s := U(\lie a)_s \cdot m$ and the associated graded module $\text{gr } M$ is a cyclic module for $S(\lie n^-)$.

We denote by $W$, the Weyl group of $\lie g$, it is generated by the simple reflections $s_i$ corresponding to the simple roots $\alpha_i$. 
The Weyl group acts on the weights in $P$ and on the weights of $V(\lambda)$, leaving the dimension of the weight spaces invariant. 
For $w\in W$, the extremal weight space of weight $w(\lambda)$ in $V(\lambda)$ is therefore one-dimensional and we denote $v_{w(\lambda)}$ a basis element in this weight space. \\
Let $\lie b = \lie n^+ \oplus \lie h$ be a Borel subalgebra of $\lie g$. 
The Demazure module corresponding to $w\in W$ is defined to be the $\lie b$-module $V_w(\lambda):= U(\lie b)\cdot v_{w(\lambda)}$ generated by the element $v_{w(\lambda)}$. 
Notice that when $w$ is the longest element in $W$, the vector space $V_w(\lambda)$ coincides with $V(\lambda)$.\\

Now let $\lie g$ be of type $X$ and $\tau$ an automorphism of the Dynkin diagram, $\lie g^{\tau}$ be the fixed point algebra of type $X^{\tau}.$ 
Let $R_{ X}^-$ and $R_{X^{\tau}}^-$ denote negative roots of $\lie g$ and $\lie g^{\tau}$ respectively. In particular, we will consider the following automorphisms and fixed point algebras. The automorphism of type:
\begin{enumerate}
    \item[(a)] ${\tt A}_{2n-1}$, and corresponding fixed point algebra which is of type {\tt C}$_n$.
    \item[(b)] ${\tt D}_{n+1}$, and corresponding fixed point algebra being of type {\tt B}$_n$.
\end{enumerate}
\bigskip

We introduce two partial orders on $R^-$. The first one is classical:
\[
\beta < \gamma :\Leftrightarrow \gamma - \beta \in \sum \mathbb{Z}_{\geq 0} \alpha_i.
\]
Let $\delta \in R^- + R^-$, then
\[
\mathcal{P}(\delta) := \begin{cases} (\delta, 0) \quad\text{ if } -\delta \text{ is the highest root of a Dynkin subdiagram, } \\ \{(\beta_1, \beta_2) \in (R^- \cup \{ 0\})^2 \mid \beta_1 + \beta_2 = \delta \} \,\, \text{ else}. \end{cases}
\]
We define a partial order $\leq$ on $\mathcal{P}(\delta)$ by
\[
(\beta_1, \beta_2) \leq (\gamma_1, \gamma_2) :\Leftrightarrow \gamma_1 \leq \beta_1, \beta_2 \leq \gamma_2.
\]

For clarity, we include here a short example of the above partial order on $\mathcal{P}(\delta)$. In type {\tt A}$_3$, consider the pair of negative roots $(\alpha_{1,2}, \alpha_{2,3})$, then one has $(\alpha_{1,3}, \alpha_{2,2}) \in \mathcal{P}(\alpha_{1,2}+ \alpha_{2,3})$, and $(\alpha_{1,2}, \alpha_{2,3}) < (\alpha_{1,3}, \alpha_{2,2})$. Consider type {\tt C}$_3$ and the pair $(\alpha_{1,\overline{2}},\alpha_{1,3})$, then $(\alpha_{1,\overline{2}},\alpha_{1,3}) < (\alpha_{1,\overline{1}},\alpha_{2,3})\in \mathcal{P}(\alpha_{1,\overline{2}}+\alpha_{1,3})$.

\begin{definition}
Let $(\beta_1, \beta_2) \in R^- \times R^-$, then we define the $\textbf{join}$ of $\beta_1$ and $\beta_2$:
\[
\beta_1 \oplus \beta_2 := \{(\gamma_1, \gamma_2) \mid (\gamma_1, \gamma_2) \in \mathcal{P}(\beta_1 + \beta_2) \mid (\beta_1 , \beta_2) < (\gamma_1, \gamma_2)\}.
\]
By abuse of notation, we denote $\beta_1 \oplus \beta_2$ also the set of roots in $\beta_1 \oplus \beta_2$.
\end{definition}

\begin{example}\label{exam-join}
To explain the join of two roots, we work it out for types {\tt A} and {\tt C}:
\begin{enumerate}
    \item Suppose $\beta_1 = \alpha_{i_1, i_2}, \beta_2 = \alpha_{j_1, j_2}$ with $i_1 \leq j_1, i_2, j_2 \in I$, then: 
    \[
        \beta_1 \oplus \beta_2 = \begin{cases}  \emptyset & \text{ if } i_1 = j_1 \text{ or } i_2 < j_1-1 \text{ or } j_2 \leq i_2, \\
        \{ (\alpha_{i_1, j_2}, \alpha_{j_1, i_2} )\} & \text{ if } j_1 -1 \leq i_2 < j_2.
    \end{cases}
    \]
    \item Suppose $\beta_1 = \alpha_{i_1, \overline{i}_2}, \beta_2 = \alpha_{j_1, \overline{j}_2}$ with $i_1 \leq j_1, i_2, j_2 \in I$, then:
    \[
        \beta_1 \oplus \beta_2 = \begin{cases}  \emptyset & \text{ if } i_2 < j_1, \\ 
        \{ (\alpha_{i_1, \overline{j}_1}, \alpha_{j_2, \overline{i}_2}), (\alpha_{i_1, \overline{j}_2}, \alpha_{j_1, \overline{i}_2})\} & \text{ if } j_1 \leq j_2 < i_2,\\
        \{ (\alpha_{i_1, \overline{j}_1}, \alpha_{i_2, \overline{j}_2}) \} & \text{ if } j_1 \leq i_2 \leq j_2.\\
    \end{cases}
    \]
    \item For $\beta_1 = \alpha_{i_1, i_2}, \beta_2 = \alpha_{j_1, \overline{j}_2}$ with $i_1\leq j_1, i_2, j_2 \in I$, one has:
    \[\beta_1 \oplus \beta_2 = \begin{cases}  \emptyset & \text{ if } i_1 \leq i_2 < j_1-1\leq j_2, \\ 
    \{(\alpha_{i_1, \overline{j}_1}, \alpha_{j_2,i_2}),(\alpha_{i_1, \overline{j}_2}, \alpha_{j_1, i_2})\} & \text{ if } j_1 \leq j_2 \leq i_2,\\
        \{(\alpha_{i_1, \overline{j}_2}, \alpha_{j_1, i_2}), (\alpha_{i_1,\overline{j}_1}, 0)\} & \text{ if } j_1 \leq i_2 = j_2-1,\\
        \{(\alpha_{i_1, \overline{j}_2}, \alpha_{j_1, i_2})\} & \text{ if } j_1 \leq i_2 < j_2-1,\\
        \{(\alpha_{i_1,\overline{j}_2}, 0)\} & \text{ if } i_1 \leq i_2 = j_1-1\leq j_2. 
    \end{cases}
    \]
    
\item Lastly, for $\beta_1 = \alpha_{i_1, i_2}, \beta_2 = \alpha_{j_1, \overline{j}_2}$ with $j_1 < i_1, i_2, j_2 \in I$, we have:
    \[\beta_1 \oplus \beta_2 = \begin{cases}  \emptyset & \text{ if } i_1 \leq i_2 < j_2-1, \\ 
    \{(\alpha_{j_1, \overline{i}_1}, \alpha_{j_2, i_2})\} & \text{ if } i_1 \leq j_2 \leq i_2 \text{ or } j_1 \leq j_2 < i_1 \leq i_2,\\
        \{(\alpha_{j_1, \overline{i}_1}, 0)\} & \text{ if } i_1 \leq i_2 = j_2-1.
    \end{cases}
    \]
\end{enumerate}
\end{example}
\bigskip

The second order on $R^-$ is defined as follows: \\
Let $\alpha = \sum a_i \alpha_i, \beta = \sum b_i \alpha_i \in R^-$. Let $s = \text{min} \{i \mid a_i \neq 0\}$. We define $\alpha \succ \beta$ if 
\[
\beta = \alpha + \alpha_s \text{ and } a_s = 1
\] 
or
\[
\beta = \alpha - \alpha_j \text{ for some } j  \geq s.
\] 
The transitive closure of these cover relations defines an order $\prec$ on $R^-$. 
\begin{example}
Two examples to explain the order:
\begin{enumerate}
    \item In type ${\tt A}$: $\alpha_{i,j} \succ \alpha_{k,\ell} \Leftrightarrow i \leq k, j \leq \ell$.
    \item In type ${\tt C}$: $\alpha_{i,\overline{j}} \succ \alpha_{k,\overline{\ell}} \Leftrightarrow i \leq k, j \geq \ell$.
\end{enumerate}
\end{example}
We note here, that if $\beta_1 \succ \beta_2$, then $\gamma_1$ and $\gamma_2$ are non-comparable with respect to $\succ$ for all $(\gamma_1, \gamma_2) \in \beta_1 \oplus \beta_2$. The motivation for the second order is due to the definition of a Dyck path for Lie algebras of type ${\tt A}$ or ${\tt C}$ from \cite{FFL11a, FFL11b}:
\begin{definition}
A Dyck path $\mathbf{p}$ is a sequence of positive roots $(\beta_1, \ldots, \beta_s)$ such that
\begin{itemize}
    \item $\beta_1$ is a simple root $\alpha_i$, we denote $s(p) = i$.
    \item $\beta_s$ is either a simple root $\alpha_i$ or additionally in type {\tt C} equals $\alpha_{i,  \overline{i}}$, we denote $e(p) = i$ or $e(p) = n$ (in the latter case).
    \item For all $1 \leq i < s$: $\beta_{i} \succ \beta_{i+1}$.
\end{itemize}
\end{definition}

\section{FFLV-admissible sets and fixed point subalgebras}\label{sec:admissible}
We will introduce the main combinatorial object of our study. We will explain in Section~\ref{sec:application}, how this definition is motivated by the study of certain monomial bases of simple, finite-dimensional modules.
\begin{definition}\label{def:admissible}
A subset $A \subseteq R^-$ is called \textbf{FFLV-admissible} (short \textbf{admissible}) if 
\begin{enumerate}
    \item $\lie n_A = \bigoplus_{\alpha \in A} \lie g_{\alpha}$ is a Lie subalgebra.
    
    \item For all pairs $\alpha \succeq \beta \in A$ one has $\alpha \oplus \beta \in A$.
\end{enumerate}
\end{definition}

We note, that if $A$ is admissible, then 
 $$\lie g_{A} := \lie h \bigoplus_{\alpha \in A} \lie g_\alpha$$ is a Lie subalgebra in $\lie g$. 

\begin{example}\label{exa:admissible}
We provide a short list of examples here, and refer to more example classes  to Section~\ref{sec:triangular}:
\begin{enumerate}
    \item There are two trivial admissible subsets, $\emptyset$ and $R^-$.
    \item The next obvious examples to be considered are subsets with one root $A = \{\alpha\}$. 
    Then $A$ is admissible if and only if $\alpha$ is the highest root of the subdiagram:
    \begin{enumerate}
        \item In type ${\tt A}$, every subset $\{ \alpha \}$ is admissible.
        \item In type ${\tt C}$, $\{\alpha\}$ is admissible if and only if $\alpha = \alpha_{i,j}$ with $j < n$ or $\alpha = \alpha_{i, \overline{i}}$.
    \end{enumerate}
    \item In type ${\tt C}$, the minimal admissible set containing $\alpha_{i, \overline{j}}$ with $i \neq j$ can be read off from Example~\ref{exam-join}, it is $\{\alpha_{i, \overline{j}}, \alpha_{i, \bar{i}}, \alpha_{j,\bar{j}}\}$.
    \item Slightly larger, non-trivial examples are the subsets given by the roots appearing in the following grid for type ${\tt A}_5$ and type ${\tt C}_3$ respectively.\\

\begin{minipage}{0.45\textwidth}
{\tiny\begin{tikzcd}
\bullet\arrow[r]&\alpha_{1,2}\arrow[d]\arrow[r]&\bullet\arrow[d] \arrow[r]&\alpha_{1,4}\arrow[d]\arrow[r]&\alpha_{1,5}\arrow[d]\\
                     &\alpha_{2,2}         \arrow[r]&\bullet\arrow[d] \arrow[r]&\alpha_{2,4}\arrow[d]\arrow[r]&\alpha_{2,5}\arrow[d]\\
                     &                              &\bullet        \arrow[r]&\bullet\arrow[d]\arrow[r]&\bullet\arrow[d]\\
                     &                              &                               &\alpha_{4,4}         \arrow[r]&\alpha_{4,5}\arrow[d]\\
                     &                              &                               &                              &\bullet
\end{tikzcd}}
\end{minipage}
\begin{minipage}{0.4\textwidth}
{\tiny\begin{tikzcd}
\bullet\arrow[r] & \alpha_{1,2}\arrow[d] \arrow[r] & \bullet\arrow[d] \arrow[r]& \alpha_{1,\overline{2}}\arrow[d]\arrow[r]& \alpha_{1,\overline{1}}\\
& \alpha_{2,2} \arrow[r] & \arrow[d] \bullet\arrow[r]& \alpha_{2,\overline{2}}\\
                                  & & \bullet
\end{tikzcd}}
\end{minipage}
\end{enumerate}
\end{example}

\begin{proposition}
Let $A \subset R^-$ be admissible and $S\subset A$. Let $\supp S = \bigcup_{\alpha \in S} \supp \alpha$ and $S_i \subset S$ be a connected component. Then $\theta_{S_i}$, the highest root of the Dynkin subdiagram of $S_i$ belongs to $A$.
\end{proposition}
\begin{proof}
The statement follows for roots of the form $\alpha_{i,j}$ with $i,j \in I$ straight from the definition. Now, let $\beta \in A$ and suppose $\beta \oplus \beta \neq \emptyset$. Let $S = \supp \beta$, then $\theta_S \in \beta \oplus \beta$. 
\end{proof}

\subsection{Fixed point subalgebras and admissible sets}
We discuss here how admissible sets and fixed point sets of Dynkin automorphisms are compatible. Recall the Lie algebra $\lie g$ of type $X$ and the corresponding fixed point algebra $\lie g^{\tau}$ of type $X^{\tau}$ for a fixed Dynkin automorphism $\tau$. We have two notions of admissibility for $\lie g^{\tau}$: the first one is given by Definition \ref{def:admissible}. For the second one, we denote for a $\tau$-invariant $A \subset R^{-}_{X}$ the set of orbits by $A^{\tau} \subset R^{-}_{X^{\tau}}$. Recall that every subset of $R^{-}_{X^{\tau}}$ is of the form $A^{\tau}$ for a unique $A\subset R^{-}_{X}$.

\begin{definition}
We call $A^{\tau} \subset R^{-}_{X^{\tau}}$ \textbf{induced admissible} if $A \subset R^{-}_{X}$ is admissible.
\end{definition}

The following proposition gives the equivalence of the two notions:
\begin{proposition}\label{prop-adm-ind}
Let $A^{\tau} \subset R_{X^{\tau}}^-$, then $A^{\tau}$ is admissible if and only if $A^{\tau}$ is induced admissible.
\end{proposition}

\begin{proof}
Let $A = \tau(A)$ be admissible and $\beta_1^{\tau}, \beta_2^{\tau} \in A^{\tau}$ with $\beta_1^{\tau} \succeq \beta_2^{\tau}$. We choose minimal representatives $\beta_1, \beta_2$ with respect to $\succ$. Suppose $(\gamma_1^{\tau}, \gamma_2^{\tau}) \in \beta_1^{\tau} \oplus \beta_2^{\tau}$, e.g. $\gamma_1^{\tau} + \gamma_2^{\tau} = \beta_1^{\tau} + \beta_2^{\tau}$ and $\gamma_1^{\tau} \leq \beta_1^{\tau}, \beta_2^{\tau} \leq \gamma_2^{\tau}$. We choose minimal representatives $\gamma_1, \gamma_2$, then 
\[
\gamma_1+ \gamma_2 = \beta_1 + \beta_2 \text{ and } \gamma_1 \leq \beta_1, \beta_2 \leq \gamma_2
\]
and hence $(\gamma_1, \gamma_2) \in  \beta_1 \oplus \beta_2$. $A$ is admissible and and $\tau$ invariant, so $(\gamma_1^{\tau}, \gamma_2^{\tau}) \in \beta_1^{\tau} \oplus \beta_2^{\tau}$.\\
Suppose $A = \tau(A)$ and $A^{\tau}$ is admissible. Similarly as before, one shows for $\beta_1, \beta_2 \in A$ that $\beta_1 \oplus \beta_2 \subset A$ if both are the minimal or both are the maximal elements in their orbit. We are left with $\beta_1$ minimal and $\beta_2$ maximal, $\beta_1 \succ\beta_2$ and suppose $(\gamma_1, \gamma_2) \in \beta_1 \oplus \beta_2$. Then $\gamma_1^{\tau} + \gamma_2^{\tau} = \beta_1^{\tau} + \beta_2^{\tau}$ and $\gamma_1^{\tau} \leq \beta_1^{\tau}, \beta_2^{\tau} \leq \gamma_2^{\tau}$ as those properties are true for $\gamma_1, \gamma_2$. Since $A^{\tau}$ is admissible, $\{\gamma_1^{\tau}, \gamma_2^{\tau}\} \in A^{\tau}$ and so $\{\gamma_1, \gamma_2\} \in A$.
\end{proof}

\subsection{Saturated sets}
\begin{definition}
For a subset $S \subset R^-$, we define the \textbf{join closure} $\overline{S}^{\oplus}$ of $S$ to be the minimal subset in $R^-$, containing $S$ and for any two roots $\alpha, \beta \in \overline{S}^{\oplus}$, $\alpha \oplus \beta \subset S$.
\end{definition}
This is certainly well-defined and a non-trivial subset, for example in type ${\tt A}$, if $S= \{\alpha \}$, then $S = \overline{S}^{\oplus}$ while the closure of the set of all (negative) simple roots is $R^-$. Admissible subsets are by definition closed under join.\\
We consider the join closure of particular subsets:
\begin{proposition}\label{prop-a-grid}
Let $\lie g = \lie{sl}_{n+1}$ and $\mathbf{p} = (\alpha_{i_1,j_1}, \ldots, \alpha_{i_s, j_s})$ a Dyck path, such that consecutive roots are linked. Then \[
\overline{\mathbf{p}}^{\oplus} = \{ \alpha_{i_k, j_\ell} \mid 1 \leq k, \ell \leq n \}.
\]
\end{proposition}
\begin{proof}
Let $1 \leq k < n$, then
\[
\alpha_{i_k, j_k} \oplus \alpha_{i_{k+1}, j_{k+1}} = \{ \alpha_{i_k, j_{k+1}}, \alpha_{j_k, i_{k+1}} \}.
\]
Iterating this procedure gives the proposed set. We are left to show, that this is closed under join: consider $\alpha_{i_k, j_\ell}, \alpha_{i_r, j_t}$ with $i_k \leq i_r$ and $i_k, i_r \in \{i_1, \ldots, i_s\}, j_\ell, j_t \in \{j_1, \ldots, j_s\}$. Then by Example~\ref{exam-join} we see that
\[
\alpha_{i_k, j_\ell} \oplus \alpha_{i_r, j_t} = \emptyset
\]
unless $i_r - 1 \leq j_\ell < j_t$, while in this case
\[
\alpha_{i_k, j_\ell} \oplus \alpha_{i_r, j_t} = \{ \alpha_{i_k, j_t}, \alpha_{j_\ell, i_r}\}
\]
which is in the proposed set.
\end{proof}

\begin{proposition}\label{prop-c-grid}
Let $\lie g = \lie{sp}_{2n}$ and $\mathbf{p}$ be a Dyck path, such that consecutive roots are linked, then 
\[
\overline{\mathbf{p}}^{\oplus} =  (\overline{T_{\mathbf{p}}}^{\oplus})^{\tau} \text{ where } 
T_{\mathbf{p}} = \bigcup_{\alpha^{\tau} \in \mathbf{p}} \{\gamma  \in \alpha^{\tau} \} \subset R^-_{{\tt A}_{2n-1}}.
\]
\end{proposition}
\begin{proof}
Let $\alpha^{\tau}, \beta^{\tau} \in R^-$, then we can read off from Example~\ref{exam-join}:
\[
\alpha^{\tau} \oplus \beta^{\tau} = \{ \gamma_1 \oplus \gamma_2 \mid \gamma_1 \in \alpha^{\tau}, \gamma_2 \in \beta^{\tau}\}^{\tau}.
\]
Applying this to $\mathbf{p}$ yields
\[
(\overline{T_{\mathbf{p}}}^{\oplus})^{\tau} \subseteq \overline{\mathbf{p}}^{\oplus}
\]
since by construction $\tau(\overline{T_{\mathbf{p}}}^{\oplus}) = \overline{T_{\mathbf{p}}}^{\oplus}$.\\ We are left to show that the left hand side is closed under join. 
So let $\alpha^{\tau}, \beta^{\tau} \in (\overline{T_{\mathbf{p}}}^{\oplus})^{\tau}$, then the first equation in the proof shows that $\alpha^{\tau} \oplus \beta^{\tau} \subset (\overline{T_{\mathbf{p}}}^{\oplus})^{\tau}$, because $\overline{T_{\mathbf{p}}}^{\oplus}$ is closed under join.
\end{proof}
The following will be a crucial observation in the proof of the main Theorem~\ref{thm:main}.
\begin{proposition}\label{prop-diff}
Let $\lie g = \lie{sl}_{n+1}, \lie{sp}_{2n}$ and let $\mathbf{p}$ be a Dyck path, such that consecutive roots are linked. Let $\alpha, \beta \in \overline{\mathbf{p}}^{\oplus}$ and suppose $\gamma = \alpha - \beta \in R^+$. Then 
\[
(\overline{\mathbf{p}}^{\oplus} + \gamma) \cap R^- \subset \overline{\mathbf{p}}^{\oplus}.
\]
\end{proposition}
\begin{proof}
For $\lie g = \lie{sl}_{n+1}$ and $\mathbf{p} = (\alpha_{i_1,j_1}, \ldots, \alpha_{i_s, j_s})$, $\gamma$ could be only one of the following
\[
\{ \alpha_{i_k, i_{\ell} -1}, \alpha_{j_k -1, j_\ell} \mid k < \ell \}.
\]
Observing, that $\overline{\mathbf{p}}^{\oplus}$ is the set of all roots starting in $i_k$ and ending in $j_\ell$ completes the proof in this case.\\
Turning to $\lie{g} = \lie{sp}_{2n}$, we again observe that $\overline{\mathbf{p}}^{\oplus}$ is admissible and hence induced admissible. Any difference of two roots in $\overline{\mathbf{p}}^{\oplus}$ is in fact the orbit of a difference between two roots in $\lie{sl}_{2n}$. We can therefore reduce the stability of $\overline{\mathbf{p}}^{\oplus}$ to the $\lie{sl}_{n+1}$-case.
\end{proof}

\section{Triangular Weyl group elements}\label{sec:triangular}
In this section, we consider particular subsets $A \subset R^+$. Let $w \in W$, then
\[
A_w := w^{-1}(R^+) \cap R^-.
\]
We are aiming to classify the $w \in W$, such that $A_w$ is admissible.

\begin{definition}
Let $\mathbf{i} := i_1 i_2 i_3 i_4$ with $i_j \in \{1, \ldots, 4\}$ pairwise distinct. We call $\tau \in S_{N}$ $\mathbf{i}$-avoiding (\textit{pattern avoiding}) if there is no quadruple $j_1 < j_2 < j_3 < j_4$ such that $ (\tau(j_1), \tau(j_2), \tau(j_3), \tau (j_4))$ is ordered as $\mathbf{i}$.
\end{definition}

\subsection{Type \texorpdfstring{\tt A}{A}}
The case of admissible sets originating from Weyl group elements in type {\tt  A} has been treated in \cite{Fou16}. We recall here the original definition from \cite{Fou16} and \cite{CFF20}:

\begin{definition}
Let $w \in S_{n+1}$, then $w$ is called \textbf{triangular} if $w$ does not contain patterns of the form $4231$ and $2413$.
\end{definition}

We should note here some interesting facts on triangular elements in type {\tt  A}:
\begin{enumerate}
    \item The number of triangular elements is given by the sequence A032351 in OEIS, \cite{OEIS2}. The first few elements suggest, that the number might be close to $n!$ but this is not true for large $n$.
    \item There are interesting elements in $S_n$ introduced in \cite{DY01}, which are also triangular elements as discussed in \cite{Fou16}.
    \item The ``smallest'' non-triangular element is $s_1s_2s_3$ in $S_4$.
    \item There is a subset of triangular elements called rectangular elements introduced in \cite{CFF20}. These avoid the patterns $2431$ and $4213$ in addition. They are counted by the integer sequence A006012 in OEIS, \cite{OEIS1}. 
\end{enumerate}

The following can be found in \cite[Proposition 1]{Fou16} and it explains the restriction of our focus to triangular Weyl group elements.

\begin{proposition}\label{pro:A-triangular}
Let $w \in S_{n+1}$, then $w^{-1}(R^+) \cap R^-$ is admissible if and only if $w$ is triangular.
\end{proposition}

\subsection{Type \texorpdfstring{\tt C}{C}}
We turn to the symplectic case and it's Weyl group $W_{{\tt C}_n}$. We consider the symmetric group $S_{2n}$ with the involutive element 
\[
w_0: \{1, \ldots, 2n\} \longrightarrow \{ 1, \ldots, 2n \} \; , \; i \mapsto 2n+1-i.
\]
With the help of $w_0$ we define an involution on $S_{2n}$ by
\[
\sigma: S_{2n} \longrightarrow S_{2n}\; , \; \tau \mapsto w_0 \circ \tau \circ w_0.
\]
The symplectic Weyl group $W_{{\tt C}_n}$ is generated by $s_1, \ldots, s_n$ subject to the relations
\[
s_i^2 =1, (s_is_j)^2 = 1 \text{ for } |i-j| > 1, (s_is_{i+1})^3 = 1 \text{ for } i < n-1, (s_{n-1} s_n)^4 = 1.
\]
This symplectic Weyl group is isomorphic to the subgroup of $\sigma$-invariant elements in $S_{2n}$, that is $\tau =  w_0 \circ \tau \circ w_0$, via the following identification $\iota$
\[s_i \mapsto t_it_{2n-i} \text{ for } i \leq n-1; s_n \mapsto t_n.\]
Again we are interested in the $w \in W_{{\tt C}_n}$ such that $A_w$ is admissible. Our goal is
\begin{lemma}\label{lem-adm-adm}
There is a bijection from the set $\{ w \in W_{{\tt C}_n} \mid A_w \text{ is admissible } \}$ to the set $\{ w \in W_{{\tt A}_n} \mid A_w \text{ is admissible }\}$.
\end{lemma}
Before giving the proof, we use the lemma to define
\begin{definition}\label{def:symplectictriangular}
An element $w \in W_{{\tt C}_n}$ is called \textbf{symplectic triangular} if $\iota(w)$ is triangular for type {\tt  A}. 
\end{definition}
Due to the facts on triangular elements, we deduce
\begin{corollary}
The number of symplectic triangular elements is given by the sequence A032351 in OEIS, \cite{OEIS2}. 
\end{corollary}

One could ask whether there is also bijection of the admissible sets of $R_{{\tt C}_n}$ to the admissible sets of $R_{{\tt A}_n}$. In turns out, that there are $8$ admissible subsets in type ${\tt C}_2$ and only $7$ admissible subsets in type ${\tt A}_2$; in type ${\tt A}_2$ all Weyl group elements are triangular while in type ${\tt C}_2$, the list of symplectic triangular Weyl group elements is
\[
\text{id}, s_1, s_2, s_1s_2, s_2s_1s_2, s_1s_2s_1s_2.
\]
In type ${\tt A}_2$, there is one admissible set which is not obtained from a Weyl group element, namely $\{\alpha_1 + \alpha_2\}$, while in type ${\tt C}_2$, there are two $\{\alpha_{1,\bar{1}}\}$ and $\{\alpha_2, \alpha_{1, \bar{1}}\}$.

We turn to the proof of Lemma~\ref{lem-adm-adm}, which has three steps:
 \begin{proposition}\label{pro:patternavoiding}
Fix a pattern $\mathbf{i} \subset \{ 1, \ldots, 4\}$. There is a mapping between the elements from $S_{2n}$ that avoid $\mathbf{i}$, and the elements from $S_{n+1}$ that avoid $\mathbf{i}$.
 \end{proposition}
\begin{proof}
Let $\tau =  (j_1, \ldots, j_{2n}) \in S_{2n}$ and define $\tilde{\tau}$ to be the sequence where all $j > n+1$ are replaced by $n+1$. Let $\underline{\tau}$ denote the sequence where all $n+1$ in $\tilde{\tau}$ are deleted except the leftmost. Then $\underline{\tau} \in S_{n+1}$. It follows immediately, that if $\tau$ avoids the pattern $\mathbf{i}$, then $\underline{\tau}$ avoids $\mathbf{i}$ as well.
\end{proof}

In the next step, we prove the surjectivity:
\begin{proposition}\label{prop-surj}
Consider a pattern $\mathbf{i}^* := i_1 i_2 i_3 i_4$ with $i_j \in \{1, \ldots, 4\}$ pairwise distinct such that $i_1>i_2,i_4$ and $i_3 > i_2,i_4$. Then there is a surjection between the elements in $S_{2n}$ that are $\sigma$-invariant and avoid the pattern $\mathbf{i}^*$ and the elements in $S_{n+1}$ that avoid the pattern $\mathbf{i}^*$.
\end{proposition}
\begin{proof}
Suppose $\tau\in S_{2n}$ is invariant under $\sigma$ and avoids $\mathbf{i}^*$. Let $\underline{\tau}$ be the pattern constructed from $\tau$ as in the proof of Proposition \ref{pro:patternavoiding}. Let $1\leq j_1 < j_2 < j_3 < j_4 \leq n+1$ be such that $\underline{\tau}(j_1, j_2, j_3, j_4)$ is the pattern $\mathbf{i}^*$. Then either $j_4 < n+1$, in which case this $\tau$ would have the same pattern, which is a contradiction or $j_4 = n+1$ and we set $n+1 \leq s \leq 2n$ the leftmost element in $\tau$ greater or equals $n+1$. Then  $\tau(j_1, j_2, j_3, s)$ is the pattern $\mathbf{i}^*$. We conclude that $\underline{\tau}$ avoids the pattern $\mathbf{i}^*$.
\end{proof}
For the bijection of Lemma~\ref{lem-adm-adm}, it remains to prove the injectivity:

\begin{proof}
Let $w\in S_{n+1}$ be an element that avoids the patterns $3142$ and $4231$. We will construct the preimage of this element in the following way. Consider the position of $n+1$ in the presentation of $w$. Let $k$ denote the number of elements left of $n+1$ in this presentation. Let $j$ be the element at position $i < n+1$. Construct an element $w'$ as follows. We leave every element $j$ at position $i$ and every element of the form $2n+1-j$ (left of $n+1$), we put at position $2n+1-i$. Now we are left with the element $n+1$ and the elements to its right. These are mapped in the following way. The ordering of the elements $1,\ldots,n$ determines the ordering of $n+1,\ldots, 2n$, because of the involution. Now replace all $n+1$ by the elements $n+1, \ldots, 2n$ which are not placed before, in that order. These are $n-k$ elements. So the elements to the right of $n+1$ are just shifted by $n-k$ to the right. 
\end{proof}

From Proposition~\ref{prop-adm-ind} and Lemma~\ref{lem-adm-adm}, we can deduce the following analogue to Proposition \ref{pro:A-triangular}:

\begin{corollary}\label{pro:c-triang}
Let $w \in W_{{\tt C}_n}$, then $w^{-1}(R^+) \cap R^-$ is admissible if and only if $w$ is symplectic triangular.
\end{corollary}
\section{Application}\label{sec:application}
We discuss in this section the application we have in mind for defining admissible sets. For this we will recall Demazure modules, the PBW filtration and the FFLV monomial bases for simple, finite-dimensional modules.
\subsection{Submodules for admissible sets}
We fix $\lie g$ and an admissible set $A \subset R^-$, then $\lie h \oplus \lie n_A^-$ is a subalgebra and hence we can consider $U(\lie h \oplus \lie n^-_A)$ as a subalgebra in $U(\lie h \oplus \lie n^-)$. We further fix a dominant, integral weight $\lambda$ and a highest weight vector $v_\lambda \in V(\lambda)$. We define
\[
V_A(\lambda) := U(\lie n_A^-)\cdot v_\lambda.
\]
By construction, this is a $\lie h \oplus \lie n^-_A$-module but in fact, it can be extended to a module for 
\[
\lie h \oplus \lie n_A^+ \oplus \lie n^-_A
\]
where $\lie n_A^+$ is the maximal subalgebra in $\lie n^+$ such that the above is invariant under $\lie n_A^+$.\\
In the special case $A =A_w$, one has
\[
V_A(\lambda) \cong V_w(\lambda)
\]
via the twist $w: U(\lie h \oplus \lie n_A^+ \oplus \lie n^-_A) \longrightarrow U(\lie b)$. We want to analyse the modules $V_A(\lambda)$ and especially asking for a monomial basis, e.g. a set of monomials in $U(\lie n^-_A)$. Note, that in the special case of Demazure modules there are monomial bases known, while we explain in the next section that we are interested in a monomial basis defined by a homogeneous lexicographic order and such bases for Demazure modules are not known in general.


\subsection{PBW filtration and sub objects}
For a Lie algebra $\lie a$, we explain briefly that the PBW filtration on $U(\lie a)$ is compatible with the filtration on the universal enveloping algebra of a subalgebra. In the case of submodules, one has to ask for mild assumptions which are satisfied in all relevant cases.

\begin{proposition}\label{pro:inducedpbw}
Let $\lie b \subset \lie a$ be a subalgebra, then $U(\lie b)_s = U(\lie a)_s \cap U(\lie b)$, i.e., the PBW filtration is compatible with subalgebras.
\end{proposition}
\begin{proof}
It is clear, that $U(\lie b)_s \subset U(\lie a)_s \cap U(\lie b)$. We fix an ordered basis $(x_1, \ldots, x_s)$ of $\lie b$ and extend with $(y_1, \ldots, y_t)$ to a basis of $\lie a$.
We denote the induced basis of $U(\lie b)$ by $\{ x_I \mid I \in \Z_{> 0}^{s}\}$. 
$U(\lie a)$ is a free $U(\lie b)$-module with basis given by ordered monomials in $(y_1, \ldots, y_t)$ and hence a basis of $U(\lie a)$ is given by $\{ x_I y_J \mid I \in \Z_{> 0}^s, J \in \Z_{> 0}^t \}$.

Suppose now $p \in U(\lie b)_r \setminus U(\lie b)_{r-1}$, then $p = \sum_{I} c_I x_I$ with $|I| \leq r$ and at least one $I$ with $|I| = r$. Suppose $p \in U(\lie a)_{r-1}$, then $p = \sum_{I,J} c_{I,J} x_I y_J$ with $|I| + |J| \leq r-1$ and by construction there is $c_{I,J} \neq 0$ for some $J$ with $|J| \neq 0$. But the freeness of $U(\lie a)$ and the equality $\sum_{I} c_I x_I = \sum_{I,J} c_{I,J} x_I y_J$ provide a contradiction. This implies that $U(\lie b)_s = U(\lie a)_s \cap U(\lie b)$.
\end{proof}

We turn to PBW filtration of sub-modules, say $\lie b \subset \lie a$ a subalgebra and $V$ a cyclic $\lie a$-module with generator $v$. Here one can apply the PBW filtration to $$V_{\lie b}:= U(\lie b)\cdot v \subset V.$$ 
In contrast to the case of subalgebras, this filtration is in general not compatible with the filtration on $V$. Let $\lie a = \langle x,y,z \rangle_{\C}$ be the abelian three dimensional Lie algebra and $\lie b = \langle x,y \rangle_{\C}$. 
We consider the $U(\lie a)$-module defined by the ideal $(x^2, y^2, z^2, xz, yz, xy -z)$. Then the degree vector of $\text{gr } V$ is $(1, 3, 0, \ldots)$. 
While the submodule through $1$ generated by $\lie b$ has degree vector $(1, 2, 1, 0 \ldots)$. 
We see, that if the vanishing ideal is not homogeneous, the PBW filtration is not compatible. One needs a few more properties to have the PBW filtration on modules compatible with submodules.\\

Let for the moment $\lie g$ be a simple, finite-dimensional complex Lie algebra with triangular decomposition $\lie g = \lie n^+ \oplus \lie h \oplus \lie n^-$ and let $\lie b \subset \lie n^-$ be a subalgebra such that $\lie b = \bigoplus_{\alpha \in R^-} \lie b \cap \lie g_{\alpha}$. Let $\lambda$ be a dominant integral weight for $\lie g$ and $V(\lambda)$ the irreducible highest weight module with highest weight $\lambda$, $v_\lambda$ a highest weight generator.

\begin{proposition}\label{pro:pbwsubmodules}
With this setup, one has $\text{gr } V_{\lie b}(\lambda) \subset \text{gr } V(\lambda)$, the PBW filtration on $V(\lambda)$ is compatible with the filtration on the subalgebra.
 \end{proposition}
 \begin{proof}
We follow here the proof in \cite{CFF20}. The PBW filtration on $V(\lambda)$ is induced by the filtration on each weight space $V_\mu(\lambda)$ and for the associated graded module, it is enough to compare the degree of each monomial in root vectors. Now let $0 \neq f_{\beta_1} \cdots f_{\beta_s}\cdot v_\lambda \in V_{\lie a}(\lambda)$ with $f_{\beta_i} \in \lie b$. Suppose there exists $f_{\gamma_1} \cdots f_{\gamma_r}\cdot v_\lambda = f_{\beta_1} \cdots f_{\beta_s}\cdot v_\lambda$ for some $r < s$. From \cite{OH02}, we deduce that there are two kind of exchange relations on roots. The first one is homogeneous and for our purpose, we can ignore it here. The second is not homogeneous and reads as $\alpha + \beta = \gamma$ for some roots. Suppose now there are $i,j, k$ with $\beta_i + \beta_j = \gamma_k$. But this implies that $f_{\gamma_k} \in \lie b$, since $\lie b$ is a subalgebra. Hence we deduce that $f_{\gamma_1} \cdots f_{\gamma_r} \in U(\lie b)$ and hence the PBW filtrations are compatible.
 \end{proof}


\subsection{The FFLV bases}
In \cite{FFL11a} and \cite{FFL11b}, monomial bases of simple, finite-dimensional modules for complex Lie algebras of type {\tt A} and {\tt C} were introduced:\\
Let $\lambda = \sum m_i \omega_i \in P^+$, then for type {\tt A} (resp. type {\tt C}), the following polytope has been defined in \cite{FFL11a} (resp. \cite{FFL11b}):
 \[
 P(\lambda) = \left\{ (x_\alpha) \in \mathbb{R}^{\sharp R^+}_{\geq 0} \mid \forall\; \text{Dyck paths } \mathbf{p}, \sum_{\alpha_ \in \mathbf{p}} x_\alpha \leq m_{e(p)} - m_{s(p)} \right \}
\]
We denote $S(\lambda) = P(\lambda) \cap \mathbb{Z}^{\sharp R^+}$, the lattice points.
In the mentioned paper it is proved that
\[
\left\{ \prod_{\alpha} f_\alpha^{s_\alpha}\cdot v_\lambda \mid \mathbf{s} \in S(\lambda) \right\}
\]
is (for any chosen order in the monomials) a basis of $V(\lambda)$. This basis share an interesting property as we will explain here:
\begin{theorem}[\cite{FFL11a, FFL11b}]\label{thm-fflv} The following is a basis of $\text{gr } V(\lambda)$:
\[
\left\{ \prod_{\alpha} f_\alpha^{s_\alpha}\cdot v_\lambda \mid \mathbf{s} \in S(\lambda) \right\}.
\]
\end{theorem}
There are several monomial bases known for $V(\lambda)$, such as the basis constructed by Gelfand-Tsetlin \cite{GT50} and its generalization \cite{Mol06}, Lusztig bases \cite{Lus90}, string bases \cite{Lit98, BZ92} but none of these bases is compatible with the PBW filtration.\\
An important property of the assigned polytopes is the following
\begin{theorem}[\cite{FFL11a, FFL11b}]\label{thm-mink}
Let $\lambda, \mu \in P^+$, then $P(\lambda)$ is a lattice polytope and moreover:
\[
P(\lambda) + P(\mu) = P(\lambda + \mu) \text{ and } S(\lambda) + S(\mu) = S(\lambda + \mu),
\]
where $+$ denotes the Minkowski sum. 
\end{theorem}
\subsection{Monomial bases for submodules}
We turn to admissible sets $A \subset R^-$ and the natural projection of the polytopes $P(\lambda)$. We define
\[
\pi_A: \mathbb{R}^{| R^+|}\longrightarrow \mathbb{R}^{|A|}, \; e_\alpha \mapsto \begin{cases} e_\alpha 
& \text{ if } -\alpha \in A, 
\\0 & \text{ else, }\end{cases}
\]
and
\[
P_A(\lambda) := \pi_A(P(\lambda)) \text{ which implies } S_A(\lambda) = \pi_A(S(\lambda)).
\]
As we are considering faces of the polytope $P(\lambda)$, we have by Theorem~\ref{thm-mink}:
\begin{corollary}\label{coro-mink}
For all $\lambda, \mu \in P^+$ and every admissible subset $A$ one has
\[
P_A(\lambda) + P_A(\mu) = P_A(\lambda + \mu) \text{ and } S_A(\lambda) + S_A(\mu) = S_A(\lambda + \mu).
\]
\end{corollary} 

We list a few useful relations:

\begin{proposition}\label{prop-relations}
Let $1 \leq i \leq n$ and let $\lie g = \lie{sl}_{n+1}$ or $\lie{sp}_{2n}$. Let $1 < k \leq i \leq  j < \ell \leq n$ (resp. $\ell < n$ for type ${\tt C})$, then one has the following relation
$$
(f_{1,j}f_{k, \ell} + c f_{1, \ell}f_{k, j})\cdot v_{\omega_i} = 0 \in V(\omega_i).
$$
In type ${\tt C}$, there are more relations that will be needed in the following:\\
\begin{enumerate}
    \item for $1 < k \leq \ell, j \leq n$, one has 
$$(f_{1,j}f_{k ,\bar{\ell}} + c_1f_{1, \bar{\ell}} f_{k,j} + c_2 f_{1, \bar{k}} f_{\ell, j})\cdot v_{\omega_i} = 0 \in V(\omega_i)$$
for some $c_r \in \C  \setminus \{0\}$, where $f_{\ell, j} = 1$ for $\ell = j+1$ and $0$ for $\ell > j+1$.
    \item for $1 < k \leq \ell < j \leq n$, one has
    $$(f_{1,\bar{j}}f_{k ,\bar{\ell}} + c_1f_{1, \bar{\ell}} f_{k,\bar{j}} + c_2 f_{1, \bar{k}} f_{\ell, \bar{j}})\cdot v_{\omega_i} = 0 \in V(\omega_i)$$
    for some $c_r \in \C  \setminus \{0\}$.
\end{enumerate}
\end{proposition}
\begin{proof}
For $\lie g = \lie{sl}_{n+1}$ one has $f_{1,\ell}^2\cdot v_{\omega_i} = 0$, resp. for $\lie g = \lie{sp}_{2n}$ one has $f_{1, \ell-1}\cdot v_{\omega_i} = 0$ and $f_{1, \bar{1}}^2\cdot v_{\omega_i} = 0 \in V(\omega_i)$. 
Certainly, for any monomial $\mathbf{m} \in U(\lie g)$: $$\mathbf{m}.f_{1, \bar{1}}^2\cdot v_{\omega_i} = 0.$$
\begin{itemize}
    \item For the first relation, one sets $\mathbf{m} = e_{j+1, \ell} e_{1,k-1}$.
    \item For the second relation, one sets $\mathbf{m} = e_{1, k-1} e_{j+1, \overline{j+1}} e_{1, \ell-1} e_{1,j}$.
    \item For the third relation, one sets $\mathbf{m} = e_{1,k-1} e_{1,j-1} e_{1, \ell-1}$.
\end{itemize}
The computation is straightforward and can be found in \cite{FFL11a, FFL11b}.
\end{proof}

We give an equivalent statement of the proposition. For $\alpha \succeq \beta \in R^-$ one has
\[
f_\alpha f_\beta \cdot v_{\omega_i} +  \sum_{\mathbf{\gamma}} c_{\mathbf{\gamma}} f_{\gamma_1} f_{\gamma_2}\cdot v_{\omega_i} = 0.
\]
where the sum is over $\mathbf{\gamma} = (\gamma_1, \gamma_2) \in \mathcal{P}(\alpha +  \beta)$ with $(\alpha, \beta) \leq (\gamma_1, \gamma_2)$.

The main motivation for the paper is the following:
\begin{theorem}\label{thm:main}
Let $\lie g$ be of type {\tt A} or {\tt C}, $\lambda  \in P^+$, $A \subset R^-$ admissible, then
\[
\left\{ \prod_{\alpha \in A} f_\alpha^{s_\alpha}\cdot v_\lambda \mid \mathbf{s} \in S_A(\lambda) \right\}
\]
is a basis of $\text{gr } V_A(\lambda)$ and hence (for any chosen ordering in the monomials) a basis for $V_A(\lambda)$.
\end{theorem}
A remark before the proof: a proof for the $\lie{sl}_{n+1}$-case has been provided in \cite{Fou16} and then Proposition~\ref{prop-adm-ind} allows to reduce the study for $\lie{sp}_{2n}$ to $\lie{sl}_{n+1}$.

\begin{proof}
Let $\lambda = \sum m_i \omega_i \in P^+$. Since $S_A(\lambda) \subset S(\lambda)$,
\[
\left\{ \prod_{\alpha \in A} f_\alpha^{s_\alpha}\cdot v_\lambda \mid \mathbf{s} \in S_A(\lambda)\right \}
\]
is linearly independent, due to Theorem~\ref{thm-fflv}. It remains to prove, that this set spans $V_A(\lambda)$. Clearly, 
\[
\left\{ \prod_{\alpha \in A} f_\alpha^{s_\alpha}\cdot v_\lambda \mid \mathbf{s} \in \mathbb{Z}_{\geq 0}^{|A|} \right\}
\]
is a spanning set and we have to prove a straightening law to reduce to $S_A(\lambda)$. Let $\mathbf{s} \in \mathbb{Z}_{\geq 0}^{|A|}$ and suppose there is a Dyck path $\mathbf{p} \subset A$, such that consecutive roots are linked and
\[
\sum_{\alpha \in \mathbf{p}} s_\alpha > m_{s(\mathbf{p})}+ \ldots + m_{e(\mathbf{p})}.
\]
We will show that $f^{\mathbf{s}}\cdot v_\lambda$ is in the span of 
\[
\{ f^{\mathbf{s}}\cdot v_\lambda \mid \mathbf{s} \in S_A(\lambda) \}.
\]
First, we consider $\lie{sl}_{n+1}$, then $\mathbf{p} = (\alpha_{i_1, j_1}, \ldots, \alpha_{i_s, j_s})$ with $i_1 \leq  \ldots \leq i_s, j_1 \leq \ldots \leq j_s$. We have seen in Proposition~\ref{prop-a-grid}, that the full subgrid $\overline{\mathbf{p}}^{\oplus}$ of roots $\alpha_{i_k,j_\ell}$ is contained in $A$.\\
In \cite{FFL11a}, a homogeneous order $<$ on $\mathbb{Z}^{|R^-|}$ is defined, we consider its restriction to $\mathbb{Z}^{|A_{\mathbf{p}}|}$. It is shown using operators from $\{ e_{\beta} \mid\beta \in R^+\}$ applied on $f_{\alpha_{i_1, j_s}}^{\sum_{\alpha \in \mathbf{p}} s_\alpha}\cdot v_\lambda ( = 0 \in V(\lambda))$ that
\[
f^{\mathbf{s}}\cdot v_\lambda = \sum_{\mathbf{t} < \mathbf{s}} c_{\mathbf{t}} f^{\mathbf{t}}\cdot v_\lambda
\]
for some $c_{\mathbf{t}}$. A further inspection of the proof in \cite{FFL11a} shows, that the needed operators $e_\beta$ are exactly given by $\beta = \gamma_1 - \gamma_2$ with $\gamma_i \in \overline{\mathbf{p}}^{\oplus}$. We conclude from Proposition~\ref{prop-diff}, that each $\mathbf{t}$ on the right hand side is supported on $\overline{\mathbf{p}}^{\oplus} \subset A$. Concluding, for each violated Dyck path we have a straightening law with respect to $<$. Iterating this procedure is a finite process, showing that \[
\{ f^{\mathbf{s}}\cdot v_\lambda \mid \mathbf{s} \in S_A(\lambda) \}
\]
is a spanning set for $V_A(\lambda)$.\\
We turn to $\lie g = \lie{sp}_{2n}$. It turns out, that the same argument as for $\lie{sl}_{n+1}$ is valid here, thanks to Proposition~\ref{prop-adm-ind}. The homogeneous order for $\lie{sl}_{n+1}$ is naturally extended to $\lie{sp}_{2n}$ in \cite{FFL11b}.\\
Similarly as above, we just have to show that for each Dyck path $\mathbf{p}$, there is a set of differential operators in $U(\lie n^+)$ that are acting on a given subset $A^{\tau}_{\mathbf{p}}$ and to generate a straightening law. \cite{FFL11b} provides the straightening law and we have to check, whether this is supported on $A$ only. But this follows as for $\lie{sl}_{n+1}$ from Proposition~\ref{prop-c-grid} and, with a close inspection of the proof in \cite{FFL11b}, from Proposition~\ref{prop-diff}. 
\end{proof}

With the following lemma, we justify the discussion of admissible sets. Reasonably, we restrict ourselves to $A \subset R^-$ such that $\lie n_A^-$ is a Lie subalgebra.
\begin{lemma}\label{lem:notadmissible}
Suppose $A \subset R^-$ is not admissible, then there exists $i \in I$ such that the lattice points in the face $P_{A}(\omega_i)$ do not parametrize a basis of $V_{A}(\omega_i)$.
\end{lemma}
\begin{proof}
Suppose $A$ is not admissible, then there exists $\beta_1 \succ \beta_2$ such that $(\gamma_1, \gamma_2) \in \beta_1 \oplus \beta_2$ but $\{ \gamma_1, \gamma_2\} \notin A$. Suppose $ \supp \beta_1 \cap \supp \beta_2 = \emptyset$ but $\beta_1 \oplus \beta_2 \neq \emptyset$, then $\beta_1 \oplus \beta_2 = \{ \beta_1 + \beta_2 \}$, but this is in $A$ since $\lie n_A^-$ is a Lie subalgebra. In both cases, there exists $i$ with 
\[
f_{\beta_1} f_{\beta_2}\cdot v_{\omega_i} \neq 0 \text{ and }f_{\gamma_1} f_{\gamma_2}\cdot v_{\omega_i} \neq 0.
\]
$\gamma_1, \gamma_2$ are non-comparable with respect to $\succ$, and hence, due to the definition of Dyck paths, $e_{\gamma_1} + e_{\gamma_2} \in S(\omega_i) \setminus S_A(\omega_i)$. Proposition~\ref{prop-relations} on the other hand implies, that $f_{\beta_1} f_{\beta_2}\cdot v_{\omega_i}$ is a nontrivial linear combination involving $f_{\gamma_1} f_{\gamma_2}\cdot v_{\omega_i}$. Since $S(\omega_i)$ parametrizes a linearly independent subset
\[
f_{\beta_1} f_{\beta_2}\cdot v_{\omega_i} \notin \text{ span } \left\{
\prod_{\alpha \in A} f_\alpha^{s_\alpha}\cdot v_{\omega_i} \mid \mathbf{s} \in S_A(\omega_i)  \right\}.
\] 
\end{proof}

\subsection{Geometric interpretation}
In this section, we are giving a geometric interpretation of the results on monomial bases. The main motivation is due to the  admissible sets of the form $A_w$ for some $w \in W$. 
In this context, the module $V_A(\lambda)$ is isomorphic to the Demazure module $V_w(\lambda)$ and one obtains an action of $\lie n^+ \oplus \lie h$ on $\text{gr } V_w(\lambda)$ (resp. $w^{-1}(\lie n^+ \oplus \lie h)$ on $\text{gr }V_A(\lambda)$.

\subsubsection{\textsf{Favourable modules and flat degenerations}}  Briefly explained is a favourable module $M$, for the general setup on favourable modules we refer to \cite{FFL17b}. Let $M$ be a cyclic finite-dimensional complex vector space acted upon by a complex algebraic unipotent group $\U$ satisfying certain conditions. 
First, this module is cyclic for $U(\lie n)$, the universal enveloping algebra of $\lie n$, the nilpotent Lie algebra corresponding to $\U$. 
For a fixed basis of $\lie n$ and a fixed homogeneous ordering $\leq$ on monomials in this basis, we obtain a filtration of $M$ and the associated graded space has one-dimensional leafs only. The monomials of the basis of $\text{gr } M$ are called essential \cite{FFL17a} and we denote $\text{es } M \subset \Z^N_{\geq 0}$ the set of the exponents of these monomials. $M$ is called favourable if
\begin{itemize}
    \item there exists a convex polytope $P(M)\subset \R_{\geq 0}^N$ such that its set of lattice points, $S(M)$, coincides with the set of essential multi-exponents $es(M)$.
    \item for all $n\geq 1$, consider $U(\lie n)\cdot (m^{\otimes n}) \subset M \otimes \cdots \otimes M$ and demand $\dim U(\lie n)\cdot (m^{\otimes n})  = |n S(M)|$.
\end{itemize}

In the following, we show that the submodule $V_{A}(\lambda)\subset V(\lambda)$ is favourable. For a fixed admissible subset $A$, let $\G_{A}$ denote the corresponding connected, simply connected Lie group of the Lie algebra $\lie g_A$. Notice that $\G_{A}$ is a unipotent subgroup of $\G$, the connected, simply connected Lie group of the Lie algebra $\lie g$. We have:

\begin{lemma}\label{lem:favourable}
For $A$ an admissible set, $V_{A}(\lambda)$ is a favourable  $\G_{A}$-module for $\lambda \in P^+$ regular.
\end{lemma}
\begin{proof}
We follow the proof of \cite[Lemma 5]{Fou16}. It suffices to consider the case $\g= \lie{sp}_{2n}$. Let the roots in $R^+$ be ordered as follows:
\begin{align*}
    \alpha_{1,\overline{1}} \succ \alpha_{1,\overline{2}} \succ \alpha_{2,\overline{2}} \succ \cdots\succ \alpha_{1,\overline{n}} \succ \cdots \succ \alpha_{n,\overline{n}} \succ \alpha_{1,n-1} \succ\cdots\succ \alpha_{n-1,n-1} \succ \cdots\succ \alpha_{1,1}.
\end{align*}
We consider the restriction of this order on $A$. By \cite[Theorem 11.8]{FFL17b}, choosing the above order and taking the induced homogeneous reverse lexicographic order on monomials in $S(\lie n^-)$, implies that the $\lie{sp}_{2n}$-module $V(\lambda)$ is favourable. In particular, we have $S(\lambda)=es(V(\lambda))$. Let $\sss\in \Z_{\geq 0}^N$, with $s_{\alpha}=0$ for all $\alpha\notin A$. Suppose $\sss$ is essential for $V(\lambda)$. By definition, and through the trivial embedding $\Z_{\geq 0}^{\# A}\hookrightarrow\Z_{\geq 0}^N$, it follows that $\sss$ is essential for $V_{A}(\lambda)$. It therefore follows that $S_{A}(\lambda)\subset es(V_{A}(\lambda))$. The opposite inclusion $es(V_{A}(\lambda)) \subset S_{A}(\lambda)$ holds by reason of dimension, which gives the equality. 

Furthermore, we have an injective homomorphism of modules $V_{A}(\lambda+\mu)\hookrightarrow V_{A}(\lambda)\otimes V_{A}(\mu)$, obtained by restricting the embedding $V(\lambda +   \mu) \hookrightarrow V(\lambda) \otimes V(\mu)$. From this, and by Corollary \ref{coro-mink}, we have that for all $\lambda,\mu \in P^+$:
\[\dim U(\lie n_{A}^-)\cdot v_{\lambda}\otimes v_{\mu} =|S_{A}(\lambda+\mu)|=|S_{A}(\lambda) + S_{A}(\mu)|. \]
The proof is completed by considering the case $\lambda = \mu $, and extending by induction to the general case.
\end{proof}

We denote the projective varieties
\[
X_{A}(\lambda) := \overline{\G_A\cdot [v_\lambda]} \hookrightarrow \mathbb{P}(V(\lambda)).
\]
and 
\[
X_{A}^a(\lambda) := \overline{\G^a_{A}\cdot [v_\lambda]} \hookrightarrow \mathbb{P}\left(\text{gr } V_{A}(\lambda)\right).
\]
where $\G^a_{A}$ denotes $|A|$ copies of the additive group $\mathbb{G}_a$, which is a Lie group associated with the Lie algebra of the vector space $\lie n_A^-$ with a trivial bracket. Note that the latter and hence $\G^a_{A}$ act on $\text{gr } M$.

\begin{definition}
For $A$ admissible and $\lambda \in P^+$ regular, we call $X_{A}^a(\lambda)$ the PBW degenerate variety corresponding to $X_{A}(\lambda)$.
\end{definition}
While the definition is reasonable for all types, the following proposition depends on the existence of a suitable basis. 
\begin{proposition}\label{pro:degenerations}
Let $\lie g$ be of type ${\tt A, C}$. Let $\lambda \in P^+$ be regular, and $A$ be an admissible set, then the variety $X_{A}(\lambda)$ degenerates flatly into the PBW degenerate variety $X_{A}^a(\lambda)$ and further into a toric variety.
\end{proposition}
\begin{proof}
It follows from \cite[Theorem 8.1]{FFL17b} since the module $V_{A}(\lambda)$ is favourable (according to Lemma \ref{lem:favourable}).
\end{proof}

\begin{remark}
In view Proposition \ref{pro:degenerations}, whenever $A$ is an admissible set given as $A_w$ for some Weyl group element $w$ and $\lambda\in P^+$ regular, we obtain flat degenerations of the Schubert variety $X_w(\lambda)$ into a PBW degenerate Schubert variety $X_w^a(\lambda)$ and further into a toric variety. More particularly, we obtain flat PBW and toric degenerations of symplectic Schubert varieties, adding to the similar scenario for Schubert varieties of type {\tt A} \cite{Fou16}.
\end{remark}

\subsubsection{\textsf{PBW-semistandard tableaux bases for coordinate rings}} 
Feigin introduced in \cite{Fei11} PBW-semistandard tableaux to describe a monomial basis of the homogeneous coordinate ring of the PBW-degenerate flag variety in type ${\tt A}$, extended by the first author in \cite{Bal20} to type ${\tt C}$. Within our setup, it is natural to describe the subset of those tableaux parametrizing a basis of the homogeneous coordinate ring of the PBW degenerate Schubert variety $X_w^a(\lambda)$ whenever $w$ corresponds to an admissible subset $A$.

For a Young diagram $Y_{\lambda}$ corresponding to a partition $\lambda=(\lambda_1\geq \cdots\geq \lambda_{n} \geq 0)$, let $\mu_c$ denote the length of the $c$-th column.

A \emph{type ${\tt A}$ PBW-semistandard tableau} is a filling of the Young diagram $Y_{\lambda}$ of shape $\lambda=(\lambda_1\geq \cdots\geq \lambda_{n} \geq 0)$, with numbers $T_{r,c}\in \{1,\ldots, n+1\}$ satisfying the properties:
\begin{itemize}
    \item if $T_{r,c} \leq \mu_c$, then $T_{r,c}=r$,
    \item if $r_1<r_2$ and $T_{r_1,c}\neq r_1 $, then $T_{r_1,c}>T_{r_2,c}$,
    \item for any $c>1$ and any $r$, there exists $r'\geq r$ such that $T_{r',c-1}\geq T_{r,c}$.
\end{itemize}

A \emph{type ${\tt C}$ PBW-semistandard tableau} is a filling of the Young diagram $Y_{\lambda}$ of shape $\lambda=(\lambda_1\geq \cdots\geq \lambda_{n} \geq 0)$, with numbers $T_{r,c}\in \{1,\ldots, n, \overline{n},\ldots,\overline{1}\}$ satisfying all the three properties above and the following extra property:
\begin{itemize}
\item if $T_{r,c}= r$, and $\exists \,\, r'$ such that $T_{r',c}=\overline{r}$, then $r'<r$.
\end{itemize}

\begin{example}
Consider $\g$ of type {\tt A}$_3$. Then the full set of type {\tt A} PBW-semistandard tableaux of shape $\lambda = \omega_1 + \omega_2$ is the following set of tableaux with entries in $\{1,2,3,4\}$:\\
$$
\small\begin{ytableau}
1 & 1\\
2
\end{ytableau},
\begin{ytableau}
1 & 2\\
2
\end{ytableau},
\begin{ytableau}
1 & 1\\
3
\end{ytableau},
\begin{ytableau}
1 & 2\\
3
\end{ytableau},
\begin{ytableau}
1 & 3\\
3
\end{ytableau},
\begin{ytableau}
3 & 1\\
2
\end{ytableau},
\begin{ytableau}
3 & 2\\
2
\end{ytableau},
\begin{ytableau}
3 & 3\\
2
\end{ytableau},
\begin{ytableau}
4 & 1\\
2
\end{ytableau},
\begin{ytableau}
4 & 2\\
2
\end{ytableau},
$$
$$
\small\begin{ytableau}
4 & 3\\
2
\end{ytableau},
\begin{ytableau}
4 & 4\\
2
\end{ytableau},
\begin{ytableau}
4 & 1\\
3
\end{ytableau},
\begin{ytableau}
4 & 2\\
3
\end{ytableau},
\begin{ytableau}
4 & 3\\
3
\end{ytableau},
\begin{ytableau}
4 & 4\\
3
\end{ytableau},
\begin{ytableau}
1 & 1\\
4
\end{ytableau},
\begin{ytableau}
1 & 2\\
4
\end{ytableau},
\begin{ytableau}
1 & 3\\
4
\end{ytableau},
\begin{ytableau}
1 & 4\\
4
\end{ytableau}.
$$

By renaming the entries to read $\{1,2,\overline{2},\overline{1}\}$, it can be seen that the last four tableaux are not symplectic. So without them, we recover the respective tableaux for $\g$ of type {\tt C}$_2$.
\end{example}

\begin{proposition}[\cite{Fei12, Bal20}]\label{pro:tableaux} Let $\lie g$ be of type ${\tt A, C}$. For $\lambda =\sum_{k=1}^{n}m_k\omega_k$ a dominant integral weight, the set $\{\prod_{\alpha} f_\alpha^{s_\alpha}\cdot v_\lambda \mid \mathbf{s} \in S(\lambda)\}$ is in a weight preserving one-to-one correspondence with the set of PBW-semistandard tableaux of shape $\lambda$. 
\end{proposition}

The idea for the proof of the above proposition is the following. If a root vector $f_{\alpha_{i,j}}$ appears in $\prod_{\alpha} f_\alpha^{s_\alpha}\cdot v_\lambda$, with $\mathbf{s} \in S(\lambda)$, then there is a box in the $i$-th row of the corresponding tableau containing the entry $j+1$. Additionally for type {\tt C}, if there is a root vector of the form $f_{\alpha_{i,\overline{j}}}$ in $\prod_{\alpha} f_\alpha^{s_\alpha}\cdot v_\lambda$, with $\mathbf{s} \in S(\lambda)$, there should be an entry $\overline{j}$ in some box in the $i$-th row. Details on how these boxes are arranged to form the PBW-semistandard tableaux can be found in \cite{Bal20}.\\

Now we turn to the case of an admissible set $A$. We define

\begin{definition}
An \emph{$A$-permissible type ${\tt A}$ PBW-semistandard tableau} of shape $\lambda$ is a type ${\tt A}$ PBW-semistandard tableau of shape $\lambda$, that satisfies the following extra condition:

\begin{itemize}
    \item[(a)] for all $\alpha_{i,j}\notin A$, if $j \geq \mu_c$ and $r=i$ then $T_{r,c}\neq j+1$.
\end{itemize}

An \emph{$A$-permissible type ${\tt C}$ PBW-semistandard tableau} of shape $\lambda$  is a type ${\tt C}$ PBW-semistandard tableau of shape $\lambda$, that satisfies (a) above and the following extra condition:

\begin{itemize}
    \item[(b)] for all $\alpha_{i,\overline{j}}\notin A$, $T_{r,c}\neq \overline{j}$ for all $c$ and $r$ such that $r=i$.
\end{itemize}
\end{definition}

\begin{example}
Consider an admissible set of type ${\tt A}_3$; $A=\{\alpha_{1,2},\alpha_{2},\alpha_{1,3},\alpha_{2,3}$\}. Then the following is the full set of $A$-permissible type ${\tt A}$ PBW-semistandard tableaux of shape $\lambda = (2,1)$:
$$
\small\begin{ytableau}
1 & 1\\
2
\end{ytableau},
\begin{ytableau}
1 & 1\\
3
\end{ytableau},
\begin{ytableau}
1 & 3\\
3
\end{ytableau},
\begin{ytableau}
3 & 1\\
2
\end{ytableau},
\begin{ytableau}
3 & 3\\
2
\end{ytableau},
\begin{ytableau}
4 & 1\\
2
\end{ytableau},
\small\begin{ytableau}
4 & 3\\
2
\end{ytableau},
$$
$$
\small\begin{ytableau}
4 & 4\\
2
\end{ytableau},
\begin{ytableau}
4 & 1\\
3
\end{ytableau},
\begin{ytableau}
4 & 3\\
3
\end{ytableau},
\begin{ytableau}
4 & 4\\
3
\end{ytableau},
\begin{ytableau}
1 & 1\\
4
\end{ytableau},
\begin{ytableau}
1 & 3\\
4
\end{ytableau},
\begin{ytableau}
1 & 4\\
4
\end{ytableau}.
$$
\end{example}

The following is an analogue of Proposition \ref{pro:tableaux} for admissible sets.

\begin{proposition}\label{pro:permissibletableaux}
For an admissible set $A$, and for $\lambda =\sum_{k=1}^{n}m_k\omega_k$ a dominant integral weight, the set $\{\prod_{\alpha} f_\alpha^{s_\alpha}\cdot v_\lambda \mid \mathbf{s} \in S_A(\lambda)\}$ is in a weight preserving one-to-one correspondence with the set of $A$-permissible PBW-semistandard tableaux of shape $\lambda$. 
\end{proposition}
\begin{proof}
We claim that the bijection of Proposition \ref{pro:tableaux} still holds when restricted to the case of an admissible set $A$. Indeed, suppose $\alpha_{i,j}\notin A$ with $j \geq \mu_c$ and $r=i$. Then clearly $T_{r,c}\neq j+1$ in the corresponding tableau. And likewise for any $\alpha_{i,\overline{j}}\notin A$, it follows that $T_{r,c}\neq \overline{j}$ for all columns $c$ and rows $r$ such that $r=i$. This together with Theorem \ref{thm:main} imply the result.
\end{proof}

\printbibliography
\end{document}